\documentclass{article}

\usepackage[utf8]{inputenc}
\usepackage{fullpage}
\usepackage{mathtools}
\usepackage{amsfonts,amsmath,amssymb,amsthm}

\newtheorem{theorem}{Theorem}[section]

\newtheorem{remark}[theorem]{Remark}

\newtheorem{proposition}[theorem]{Proposition}
\newtheorem{corollary}[theorem]{Corollary}

\def\lc{\left\lceil}   
\def\rc{\right\rceil}
\def\lf{\left\lfloor}   
\def\rf{\right\rfloor}
\def\dist{\mbox {dist}}

\title{A note on panchromatic colorings}
\author{Danila Cherkashin\footnote{Saint Petersburg State University, Faculty of Mathematics and Mechanics; Moscow Institute of Physics and Technology, 
Laboratory of Advanced Combinatorics and Network Applications; St.~Petersburg Department of V.~A.~Steklov Institute of Mathematics of
the Russian Academy of Sciences.}}

\date{October 2016}

\begin{document}

\maketitle

\begin{abstract}
This paper studies the quantity $p(n,r)$, that is the minimal number of edges of an $n$-uniform hypergraph without panchromatic coloring (it means that every edge meets every color) in $r$ colors. If $r \leq c \frac{n}{\ln n}$ then all bounds have a type $A_1(n, \ln n, r)(\frac{r}{r-1})^n \leq p(n, r) \leq A_2(n, r, \ln r) (\frac{r}{r-1})^n$, where $A_1$, $A_2$ are some algebraic fractions.
The main result is a new lower bound on $p(n,r)$ when $r$ is at least $c \sqrt n$; we improve an upper bound on $p(n,r)$ if $n = o(r^{3/2})$. 


Also we show that $p(n,r)$ has upper and lower bounds depend only on $n/r$ when the ratio $n/r$ is small, which can not be reached by the previous probabilistic machinery.

Finally we construct an explicit example of a hypergraph without panchromatic coloring and with $(\frac{r}{r-1} + o(1))^n$ edges for $r = o(\sqrt{\frac{n}{\ln n}})$.
\end{abstract}

\section {Introduction}

A hypergraph is a pair $(V, E)$, where $V$ is a finite set whose elements are called vertices and $E$ is a family of subsets of $V$, called edges. 
A hypergraph is $n$-uniform if every edge has size $n$. 
A vertex $r$-coloring of a hypergraph $(V , E)$ is a map $c : V \rightarrow \{1, \dots ,r\}$. 

An $r$-coloring of vertices of a hypergraph is called \textit{panchromatic} if every edge contains a vertex of every color.
The problem of the existence of a panchromatic coloring of a hypergraph was stated in the local form by P.~Erd\H{o}s and L.~Lov{\'a}sz in~\cite{EL}.
They proved that if every edge of an $n$-uniform hyperhraph intersects at most $r^{n-1}/4(r-1)^n$ other edges then the hypergraph has a panchromatic $r$-coloring.
Then A.~Kostochka in~\cite{Kost} stated the problem in the present form and linked it with the $r$-choosability problem using ideas by P.~Erd\H{o}s, A.~L.~Rubin and H.~Taylor from~\cite{ErdRubTay}.
Also A.~Kostochka and D.~R.~Woodall~\cite{KostWood} found some sufficient conditions on a hypergraph to have a panchromatic coloring in terms of Hall ratio.
Reader can find a survey on history and results on the related problems in~\cite{Kost3, RaiSh}.

\subsection{Upper bounds}

Using the results from~\cite{Noga} A.~Kostochka proved~\cite{Kost} that for some constants $c_1$, $c_2 > 0$
\begin{equation}
\frac{1}{r} e^{c_1 \frac{n}{r}} \leq p(n,r) \leq r e^{c_2\frac{n}{r}}.
\label{1}
\end{equation}

\noindent In works~\cite{shaba1,shaba2} D.~Shabanov gives the following upper bounds:

$$p (n,r) \leq c \frac{n^2 \ln r }{r^2} \left (\frac{r}{r-1}\right)^n, \mbox{ if } 3 \leq  r = o(\sqrt n), n > n_0;$$
\begin{equation}
p (n,r) \leq c \frac{n^{3/2} \ln r }{r} \left (\frac{r}{r-1}\right)^n, \mbox{ if } r = O(n^{2/3}) \mbox{ and } n_0 < n = O(r^2);
\label{2}
\end{equation}
$$p (n,r) \leq c \max \left (\frac{n^2}{r}, n^{3/2} \right) \ln r \left (\frac{r}{r-1}\right)^n \mbox{ for all } n, r \geq 2.$$

Let us introduce the quantity $p'(n,r)$ that is the minimal number of edges in an $n$-uniform hypergraph $H = (V,E)$ such that any subset of vertices $V' \subset V$ with 
$|V'| \geq \lc \frac{r-1}{r}|V| \rc$ contains an edge. In fact $p'(n,r)$ coincides with $T(|V|, \frac{r-1}{r}|V|, n)$, where $T(a,b,c)$ stands for \textit{Tur{\'a}n number} (see~\cite{sidor} for a survey).

Note that by pigeonhole principle every vertex $r$-coloring contains a color of size at most $\lf \frac{1}{r}|V| \rf$. So the complement to this color has size at least 
$|V| - \lf \frac{1}{r}|V| \rf = \lc \frac{r-1}{r}|V| \rc$. Hence, $p(n,r) \leq p'(n,r)$. This argument is in spirit of the standard estimation of the chromatic number via the independence number.

The following theorem gives better upper bound in the case when $n = o(r^{3/2})$.

\begin{theorem} The following inequality holds for every $n \geq 2$, $r \geq 2$
$$p' (n,r) \leq c \frac{n^2 \ln r}{r} \left (\frac{r}{r-1}\right)^n.$$
It immediately implies $$p (n,r) \leq c \frac{n^2 \ln r}{r} \left (\frac{r}{r-1}\right)^n.$$
\label{theo1}
\end{theorem}

\subsection{Lower bounds}

We start by noting that an evident probabilistic argument gives $p(n,r) \geq \frac{1}{r} (\frac{r}{r-1})^n$. This gives lower bound~(\ref{1}) with $c_1 = 1$. 
This was essentially improved by D.~Shabanov in~\cite{shaba1}:
$$p(n,r) \geq c\frac{1}{r^2}  \left ( \frac{n}{\ln n} \right ) ^{1/3} \left (\frac{r}{r-1}\right)^n \mbox { for } n, r \geq 2, r < n.$$

\noindent Next, A.~Rozovskaya and D.~Shabanov~\cite{roz} showed that
$$p(n,r) \geq c\frac{1}{r^2} \sqrt {\frac{n}{\ln n}}\left (\frac{r}{r-1}\right)^n \mbox { for } n, r \geq 2, r \leq \frac{n}{2\ln n}.$$

\noindent Using the Alterations method (see Section 3 of~\cite{alon2016probabilistic}) we can get the following lower bound for all the range of $n$, $r$.
It gives better results when $r \geq c\sqrt{n}$.

\begin{theorem}
For $n \geq r \geq 2$ holds
$$p(n,r) \geq  e^{-1}\frac{r-1}{n-1}e^{\frac{n-1}{r-1}}.$$
\label{theo2}
\end{theorem}

\noindent There is a completely another way to get almost the same bound.
First, we need to prove intermediate bound. It is based on the geometric rethinking of A.~Pluh{\' a}r's ideas~\cite{Pl}.

\begin{theorem}
For $n \geq r \geq 2$ such that $r \leq c\frac{n}{\ln n}$ holds
$$p(n,r) \geq   c \max \left(\frac{n^{1/4}}{r\sqrt{r}}, \frac{1}{\sqrt n}\right) \left (\frac{r}{r-1}\right)^n.$$
\label{theo2'}
\end{theorem}

\noindent Combining Theorems~\ref{theo2} and~\ref{theo2'} we prove the following theorem.

\begin{theorem}
For $n \geq r \geq 2$ such that $\sqrt{n} \leq r \leq c'\frac{n}{\ln n}$ holds
$$p(n,r) \geq  c \frac{r}{n}e^{\frac{n}{r}}.$$
\label{theo2''}
\end{theorem}

\begin{remark}
Theorem~\ref{theo2'}, unlike Theorem~\ref{theo2} and Theorem~\ref{theo2''}, admits a local version.
\end{remark}



\subsection{Small $n/r$}

Consider the case when the ratio $n/r$ is small; $n/r = const$ is a good model case.
In the case $\frac{n}{r} \leq c \ln n$ the best upper bound was $r e^{cn/r}$~\cite{Kost}, where $c \geq 4$ is a constant. Using the following theorem we give a bound depending only on $n/r$.

\begin{theorem} The following inequality holds for every integer triple $m, n, r$
$$p (mn, mr) \leq p' (n, r).$$
\label{theo3}
\end{theorem}

\noindent As a corollary of Theorem~\ref{theo3} and an evident inequality $\max(p(n, r),  p(n+1, r+1)) \leq p(n+1, r)$ we get a better upper bound, for the case of small $n/r$.

\begin{corollary}
The following inequality holds for every integer $k \leq r$
$$p (n, r) \leq  p' \left (\lc \frac{n}{r - k + 1} \rc k, k \right).$$
In particular, if $n < r^2$ one can put $k := \alpha \frac{n}{r}$ and get $p (n,r) \leq c (\frac{n}{r})^2 \ln \frac{n}{r} \cdot e^{\frac{n}{r}}$.
\label{col}
\end{corollary}

\noindent There was no known lower bound in this case (all the previous methods give something less than $1$). Theorem~\ref{theo2} covers this gap, but note also that there exists a very simple greedy algorithm.

\begin{proposition} The following inequality holds for every integer $n \geq r$
$$p (n,r) \geq \lf \frac{n}{r} \rf.$$
\label{theo4}
\end{proposition}

\begin{proof}[Proof of Proposition~\ref{theo4}]
Consider a hypergraph $H = (V,E)$ with $|E| \leq \lf  n/r \rf$.
Let us pick an edge $e \in E$ and color its arbitrary $r$ vertices in different colors. Then let us delete $e$ and all colored vertices from $H$.
The remaining hypergraph has $|E| - 1$ edge, and the size of every edge is at least $n - r$. 
So we can do this procedure $\lf n/r \rf$ times showing the claim. 
\end{proof}

\subsection{Explicit constructions}

Recently, H.~Gebauer~\cite{GeB} gave an explicit example of an $n$-uniform hypergraph with chromatic number $r+1$ and with $(r + o(1))^n$ edges for a constant $r$.
We generalize this example to the case of panchromatic colorings.

\begin{theorem}
Let $r = o(\sqrt{\frac{n}{\ln n}})$. There is an explicit consruction of an $n$-uniform hypergraph $H = (V,E)$ without panchromatic coloring and such that
$$|E(H)| = \left( \frac{r}{r-1} + o(1)\right)^n.$$
\label{theo5}
\end{theorem}

\section {Proofs}

The following proof is just a rephrasing of the proof by P.~Erd\H{o}s~\cite{Erdos2}.

\begin{proof}[Proof of Theorem~\ref{theo1}]
Consider a vertex set $V$ of size $|V| = n^2$.
Let us construct a hypergraph $H = (V, E)$ by random (uniformly and indepentently) choosing an edge $m := c\frac{n^2\ln r}{r}  (\frac{r}{r-1} )^n$ times. 
We can choose an edge multiple times during this process, but in this case the total number of egdes can only decrease, i.~e.~$|E| \leq m$.

Let us fix a subset of vertices $V' \subset V$ of size $|V'| = \lc \frac{r-1}{r} |V| \rc$. Denote the probability that a random edge is a subset of $V'$ by $p$. Obviously,
$$p = \frac{\binom{|V'|}{n}}{\binom{|V|}{n}} = \prod_{i=0}^{n-1} \frac{\lc \frac{r-1}{r}n^2 \rc - i}{n^2 - i} \geq \left (\frac{\lc \frac{r-1}{r}n^2  \rc - n}{n^2 - n} \right )^n$$
$$\geq \left (\frac{\lc \frac{r-1}{r} n^2 \rc - 2\lc \frac{r-1}{r}n \rc}{n(n-1)}\right )^n = e \left(\frac{r-1}{r}\right)^n (1+o(1)).$$

\noindent The probability that $V'$ does not contain an edge is equal to $(1-p)^m$. The total number of such sets $V'$ is $\binom{n^2}{\lc (r-1)n^2/r \rc} = \binom{n^2}{\lf n^2/r \rf}$.
If $\binom{n^2}{\lc (r-1)n^2/r \rc} (1-p)^m < 1$ then a hypergraph realizing the inequality $p'(n,r) \leq m$ exists with positive probability.
One can see that 
$$\binom{n^2}{\lf n^2/r \rf} (1-p)^m \leq \frac{n^{2 \lf n^2/r \rf}}{\lf n^2/r \rf!} e^{-pm} = e^{c\ln r  \lf n^2/r \rf - e \left(\frac{r-1}{r}\right)^n m}.$$
So for $m = c\frac{n^2\ln r}{r}  (\frac{r}{r-1} )^n$ the claim is proved.
\end{proof}

\begin{proof}[Proof of Theorem~\ref{theo2}]
Let $H = (V, E)$ be a given hypergraph with 
$$|E| \leq  e^{-1}\frac{r-1}{n-1}e^{\frac{n-1}{r-1}}.$$
We should show that $H$ has a panchromatic coloring.

Consider an uniform independent coloring of the vertex set into $a > r$ colors.
The expectation of the number of such pairs $(e, q)$ that edge $e \in E$ has no color $q$ is $|E| a (\frac{a-1}{a})^n$. 
So, if $|E| a (\frac{a-1}{a})^n < a-r$, then with positive probability there are $r$ colors such that they are contained in every edge.
Substituting $a = \frac{(n-1)}{n-r}r$ one has that for $$|E| \leq \frac{r-1}{n-1} \left(\frac{nr-r}{nr-n} \right)^n \leq e^{-1} \frac{r-1}{n-1} e^{\frac{r-1}{n-1}}$$
a panchromatic coloring exists.
\end{proof}

\begin{proof}[Proof of Theorem~\ref{theo2'}]
Let $H = (V, E)$ be a given hypergraph with 
$$|E| \leq c \max \left(\frac{n^{1/4}}{r\sqrt{r}}, \frac{1}{\sqrt n}\right) \left (\frac{r}{r-1}\right)^n.$$
We should show that $H$ has a panchromatic coloring.

Consider an $(r-1)$-dimensional unit simplex, and let us map every vertex of $H$ to the 1-face skeleton (edges of the simplex) according to the uniform measure and independently.  
Then let us fix a bijection $f$ between colors and vertices of the simplex. We are going to color the hypergraph in the following way: for every edge $e$ of the hypergraph 
and every color $i$, we give color $i$ to the nearest (with respect to the induced metric) vertex of edge $e$ (with probability $1$ it is unique; let us call it $v_i(e)$) to the vertex of the simplex $f(i)$. If the coloring is not self-contradictory then it is obviously panchromatic.

Let us evaluate the probability of such contradiction. We are going to show that such probability is less than 1 showing the claim. 
Let us call a \textit{bad event of the first type}, the event that for some edge $e\in E$ and some color $i$ the vertex $v_i(e)$ does not lie on the adjacent to $f(i)$ edge of the simplex.
The probability of this event is $\left( \frac{r-2}{r} \right)^n$. Summing up over all edges and colors we get $Poly (r, n) \left( \frac{r}{r-1} \right)^n \left( \frac{r-2}{r} \right)^n = Poly (r, n) \left( \frac{r-2}{r-1} \right)^n$ which tends to zero if $r \leq c\frac{n}{\ln n}$. 

Now let us go to \textit{bad events of the second type}, i.\ e. the events that there is a vertex $x$ such that it should have color $i$ and $j$ simultaneously (let us call $x$ a \textit{conflict vertex}). 
Consider a pair of edges $(e_1, e_2) \in E^2$; denote the size of their intersection by $t := |e_1 \cap e_2|$. 
We will estimate the probability (denote it by $q$) that $e_1$ and $e_2$ demand to color a conflict vertex $x \in e_1 \cap e_2$ in different colors, and then sum up over all pairs of edges. The case $e_1 = e_2$ (i.\ e. $t = n$) corresponds to the event that the coloring is contradictory even on one edge $e_1$.

First, we should choose a conflict vertex $x$ (there are $t$ ways to do it) and a conflict pair of colors $(i,j)$ (there are $r(r-1)/2$ ways). Note that $x$ should lie on the 
edge $(f(i), f(j))$ of the simplex (this event has the probability $\frac{2}{(r-1)r}$), otherwise we have already counted them in the previous step. If $\dist (x, f(i)) = a$, then $\dist (x, f(j)) = 1-a$.
Since $x$ is the nearest vertex to $f(i)$ in the edge $e_1$ any vertex $y \in e_1$ cannot lie in the union of $r-1$ segments of length $a$ with endpoint $f(i)$.
Analogously, any vertex $z \in e_2$ cannot lie in the union of $r-1$ segments of length $1-a$ with endpoint $f(j)$.
So any vertex $w \in e_1 \cap e_2$ cannot lie in both forbidden sets (note that the forbidden sets have empty intersection). So for fixed $a$ the probability is 
$$\left (\frac{r-2}{r} \right)^{t-1}  \left(1 - \frac{2a}{r}\right)^{n-t} \left(1 - \frac{2-2a}{r}\right)^{n-t}.$$
Summing up, we have
$$q = t \frac{(r-1)r}{2} \frac{2}{(r-1)r}  \left (\frac{r-2}{r} \right)^{t-1} \int^1_0 \left(1 - \frac{2a}{r}\right)^{n-t} \left(1 - \frac{2-2a}{r}\right)^{n-t} da$$
$$= t \left (1-\frac{1}{(r-1)^2} \right )^{t-1} \left (\frac{r-1}{r}\right )^{2(t - 1)}\int^1_0 \left(1 - \frac{2a}{r}\right)^{n-t} \left(1 - \frac{2-2a}{r}\right)^{n-t} da.$$
Put $A :=  t e^{-tr^{-2}} > t e^{-t(r-1)^{-2}} \geq t \left (1-\frac{1}{(r-1)^2} \right )^{t} \geq \frac{1}{2} t \left (1-\frac{1}{(r-1)^2} \right )^{t-1}$. Let us show that $A \leq c \min (r^2, n)$. Indeed, $A = r^2 \frac{t}{r^2} e^{-tr^{-2}} \leq cr^2$ and $A \leq t \leq n$, so $A \leq c \min (r^2, n)$.
Put also 
$$B := \left (\frac{r-1}{r}\right )^{2(t - 1)}\int^1_0 \left(1 - \frac{2a}{r}\right)^{n-t} \left(1 - \frac{2-2a}{r}\right)^{n-t} da.$$
Obviously, $(1 - \frac{2a}{r})(1 - \frac{2 - 2a}{r}) \leq (1 - \frac{1}{r})^2$, thus $B \leq  \left  (\frac{r-1}{r}\right )^{2n - 2}.$
Exchange $x = 1-\frac{2a}{r}$ gives
$$B = \left (\frac{r-1}{r}\right )^{2(t - 1)} \frac{r}{2} \int^1_{1 - 2/r} x^{n-t} \left(2 - \frac{2}{r} - x \right)^{n-t} dx.$$
After exchange $y = \frac{1}{2}\frac{r}{r-1}x$ we have
$$B = \left (\frac{r-1}{r}\right )^{2(n - 1)} 2^{2(n-t)} \frac{r}{2} \int^{\frac{r}{2(r-1)}}_{\frac{r-2}{2(r-1)}} y^{n-t} (1-y)^{n-t} dx,$$
but this integral is not bigger than beta function 
$$B(n-t+1,n-t+1) = \frac{1}{2(n-t)+1} \frac{1}{\binom{2(n-t)}{n-t}} \leq c \frac{1}{\sqrt{n-t}} 2^{2(t-n)}.$$
Summing up, we have $B \leq c\frac{r}{\sqrt{n-t}} \left (\frac{r-1}{r} \right)^{2n}$ which implies $B \leq c \min \left(1, \frac{r}{\sqrt{n}}\right) 
\left( \frac{r-1}{r} \right)^{2n}$.
Finally, 
$$q \leq AB \leq c \min (r^2, n) \min \left(1, \frac{r}{\sqrt{n}} \right) \left (\frac{r-1}{r} \right)^{2n} = c \min \left(n, \frac{r^3}{\sqrt{n}} \right) \left (\frac{r-1}{r} \right)^{2n}.$$
The total number of such pairs $(e_1, e_2)$ is $|E|^2$, so $q|E|^2 \leq AB|E|^2 \leq \frac{1}{2}$ for a corresponding value of $c$.
Recall that the probability of bad events of the first type tends to zero, so the union bound shows the claim.

\end{proof}

\begin{proof}[Proof of Theorem~\ref{theo2''}]
Let $H = (V, E)$ be a given hypergraph with 
$$|E| \leq  c \frac{r}{n}e^{\frac{n}{r}}.$$
We should show that $H$ has a panchromatic coloring.

Put $a := r + \frac{r^2}{n}$. Since $r \leq c'\frac{n}{\ln n}$, we have $a \leq 2r \leq \frac{c}{2} \frac{n}{\ln n}$, where $c$ is from Theorem~\ref{theo2'}.
So we can repeat the proof of Theorem~\ref{theo2}. 

The probability of the union of the events of the first type still tends to zero very fast.
Now let us note that for $r \geq \sqrt{n}$ we have $\min \left(n, \frac{r^3}{\sqrt{n}} \right) = n$.
Hence the expectation of the number of such triples $(e_1,e_2,q)$ that edges $e_1$, $e_2 \in E$ conflict on color $q$ is less than
$$|E|^2cn \left (\frac{a-1}{a} \right)^{2n} = c \frac{r^2}{n}e^{\frac{2n}{r}} \left (1 - \frac{1}{r + \frac{r^2}{n}} \right)^{2n} \leq c \frac{r^2}{n}.$$
Summing up, 
$$\mathbb{E}(\# \mbox{bad triples}) \leq c \frac{r^2}{n} = \frac{a-r}{2}.$$
So by Markov inequality we have
$$\mathbb{P} (\# \mbox{bad triples} > {a-r}) \leq \frac{1}{2}.$$
It means that with positive probability there are $r$ colors such that they are contained in every edge.

\end{proof}

\begin{proof}[Proof of Theorem~\ref{theo3}]
Let $H = (V, E)$ be a hypergraph realizing the quantity $p'(n, r)$. Put $J = (W, F)$, where $W := \{(v,i) | v\in V, 1\leq i \leq m\}$, $F := \{\cup_{v \in e, 1\leq i \leq m} (v, i) | e \in E \}$. Obviously, $J$ is $mn$-uniform and $|F| = |E|$. A subset $A_v := \{(v, i) \in W | 1\leq i \leq m\}$ is called a \textit{block} (note that blocks are disjoint). 

Consider an arbitrary coloring of $|W|$ in $mr$ colors. By pigeonhole principle there is a color $i$ such that it appears in at most $\lf \frac{|W|}{mr} \rf = \lf \frac{|V|}{r} \rf$ vertices. Hence there are at most $\lf \frac{|V|}{r} \rf$ blocks with a vertex of color $i$. 
Let $V' \subset V$ be a set of such vertices $v \in V$ that the block $A_v$ does not contain color $i$.
It has the size at least $|V| - \lf \frac{|V|}{r} \rf = \lc \frac{r-1}{r}|V| \rc$, which implies the existence of an edge $e \in E$ such that $e \subset V'$. 
So the corresponding edge of $J$ does not contain color $i$, hence $p(mn, mr) \leq |F| = p'(n,r)$.
\end{proof}

\begin{proof}[Proof of Corollary~\ref{col}]

Obviously, 
$$p(n,r) \leq p \left (n, \lf \frac{r}{k} \rf k \right) \leq p \left (\lc \frac{n}{\lf r/k \rf k}   \rc \lf \frac{r}{k} \rf k , \lf \frac{r}{k} \rf k \right),$$
so by Theorem~\ref{theo3} $p (n, r) \leq p' \left (\lc \frac{n}{\lf r/k \rf k}  \rc k, k \right) \leq p' \left (\lc \frac{n}{r - k + 1} \rc k, k \right)$.

In fact, estimate (\ref{2}) is proved for $p'(n,r)$ (see~\cite{shaba1}).
So let us put $k := \alpha \frac{n}{r}$ and apply (\ref{2}). It gives $p\left(\frac{k^2}{\alpha},k \right) \leq c \frac{k^3\ln k}{k} \left (\frac{k}{k-1} \right)^{\frac{k^2}{\alpha}} = c k^2 \ln k \cdot e^{\frac{k}{\alpha}}$ showing the claim.

\end{proof}

\begin{proof} [Proof of Theorem~\ref{theo5}]
Let us construct a hypergraph $H_1 = (V_1, E_1)$ in the following way. Fix an integer $t | n$ and put $k := \lc \left (\frac{r}{r-1} \right)^t \rc \frac{n}{t}$, 
then $V := \{ (i, j) | 1 \leq i \leq k, 1 \leq j \leq rt \} = [k] \times [rt]$. Let the set of edges be
$$E := \bigcup_{A \subset [rt]} \bigcup_{\substack{0 \leq i_{\alpha} < k \\ \alpha \in A}} \bigcup_{\substack{B \subset [k] \\ |B| = \frac{n}{t}}} \left \{ \left ( (\beta + i_\alpha)\mbox{ mod } k, \alpha \right ) | \alpha \in A, \beta \in B \right \}.$$
Note that 
$$|E| \leq \binom{rt}{t} k^t \binom{k}{n/t} \leq (rt)^t \left( \left (\frac{r}{r-1} \right)^t \frac{n}{t}\right)^t \left (e \left (\frac{r}{r-1} \right)^t\right)^{n/t} \leq (rn)^t \left (\frac{r}{r-1}\right)^{t^2} e^{n/t} \left (\frac{r}{r-1}\right)^n.$$
Put $t := \sqrt{\frac{n}{\ln n}}$. Since $r = o(\sqrt{\frac{n}{\ln n}})$, one can give an estimate $(rn)^t \leq n^{2t} = e^{2t\ln n} = e^{o(n/r)}$. Also, $\left (\frac{r}{r-1}\right)^{t^2} = \left(\frac{r}{r-1}\right)^{o(n)}$ and $e^{n/t} = e^{o(n/r)}$. Summing up, $|E(H)| = \left( \frac{r}{r-1} + o(1)\right)^n$.

Let us show that $H$ has no panchromatic coloring. Suppose the contrary and consider a panchromatic coloring. 
Let us call a color $q$ a \textit{minor color} for a line $[k] \times \{i\}$ if it has at most $\lf \frac{k}{r} \rf$ vertices.
By pigeonhole principle every line $[k] \times \{i\}$ has a minor color. 
Again, by pigeonhole principle there is a set $A \subset [rt]$ of lines with the same minor color $q$ such that $|A| \geq t$. 
Next, for any fixed $\beta$ the proportion of such $\{ i_\alpha \}_{\alpha\in A}$ that $\left \{ \left ( (\beta + i_\alpha)\mbox{ mod } k, \alpha \right ) | \alpha \in A \right \}$ 
has no color $q$, is at least $\left (\frac{r-1}{r} \right)^t$. By the linearity of expectation there is a choice of $\{ i_\alpha \}_{\alpha\in A}$ such that
at least $k \left (\frac{r-1}{r} \right)^t = \frac{n}{t}$ indices $\beta \in B$ give $q$-free sets $\left \{ \left ( (\beta + i_\alpha)\mbox{ mod } k, \alpha \right ) | \alpha \in A \right \}$. So there is an edge without color $q$, which gives a contradiction.

\end{proof}

\section*{Acknowledgements} The author is supported by the Russian Science Foundation grant 16-11-10014. Also the author is grateful to Misha Basok, Roman Prosanov and Andrei Raigorodskii for very careful reading of the draft of the paper and to Fedya Petrov for some motivating remarks.

\bibliographystyle{plain}
\bibliography{main}

\end{document}